\documentclass[english]{article}
\usepackage{mathrsfs}
\usepackage{amsthm}
\usepackage{amsmath}
\usepackage{amssymb}
\usepackage{esint}

\makeatletter

\numberwithin{equation}{section}
\newtheorem{Theorem}{Theorem}[section]

\newtheorem{thm}{Theorem}[section]

\newtheorem{prop}[Theorem]{Proposition}

\newtheorem{lem}[Theorem]{Lemma}

\theoremstyle{definition}
\newtheorem{defn}[Theorem]{Definition}

\newtheorem{rem}[thm]{Remark}

  \newcommand{\imp}{\mbox{$\ \Rightarrow\ $}}
\newcommand{\Imp}{\mbox{$\Longrightarrow$}}
\newcommand{\Iff}{\mbox{$\Longleftrightarrow$}}
\newcommand{\ifg}{\mbox{$\Leftrightarrow$}}
\def\proof{{\bf Proof.}\ }
\def\ie{\emph{i.e.}}
\def\Ie{\emph{I.e.}}
\def\pes{\emph{e.g.}}
\def\Pes{\emph{E.g.}}

\def\E{{\mathcal E}}
\def\F{{\mathcal F}}

\def\Gk{{\mathfrak G}}

\def\P{{\mathcal P}}
\def\Ql{{\mathcal Q}}

\def\U{{\mathcal U}}

\def\Hl{{\mathcal H}}

\def\N{{\mathbb N}}

\def\R{{\,\mathbb R}}
\def\Z{{\mathbb Z}}

\def\WW{{\mathbb{W}}}

\def\WW{{\mathbb{W}}}
\def\II{{\mathbb{I}}}

\def\LL{{\mathbb{L}}}

\def\Zai{{\Z^\II}}
\def\Nai{{\N^\II}}

\def\Ak{{\mathfrak A}}

\def\Nk{{\mathfrak N}}

\def\ahu{{\aleph_{1}}}
\def\ag{{\alpha}}

\def\bg{{\beta}}
\def\eg{{\epsilon}}
\def\fg{{\varphi}}
\def\kg{{\kappa}}
\def\og{{\omega}}

\def\cg{{\gamma}}
\def\sg{{\sigma}}

\def\Cg{{\Gamma}}

\def\ao{{\aleph_{0}}}

\def\max{\mbox{\rm max}\;}

\def\+#1{\vec{#1}}

\def\ak{{\mathfrak{a}}}

\def\mk{{\mathfrak{m}}}

\def\zk{{\mathfrak{z}}}

\def\nk{{\mathfrak{n}}}

\def\bk{{\mathfrak{b}}}

\def\Sk{{\mathfrak{S}}}

\def\Ck{{\mathfrak{C}}}

\def\Mk{{\mathfrak{M}}}

\def\Nk{{\mathfrak{N}}}

\def\*{\times}
\def\st{{}^{\star}\!\, }
\def\0{\emptyset}
\def\7{\setminus}
\def\_{\overline}
\def\eq{\simeq}
\def\<{\prec}
\def\ll{\preceq}

\def\o+{\bigoplus}

\def\ult#1#2{^{#1}_{\; #2}}

\def\incl{\subseteq}

\def\linc{\supseteq}

\def\pincl{\subset}

\def\la{\langle}
\def\ra{\rangle}

\def\zfc{\textsf{ZFC}}

\def\qed{\hfill $\Box$}

\def\st{such\ that}

\def\Eq{$\mathsf{E4} $}

\def\PP{$(\mathsf{PP}) $}

\def\FAP{$\mathsf{FAP} $}

\def\SubP{\sf{(SubP)}}
\def\Diff{\sf{(Diff)}}
\def\CP{\sf{(CP)}}

\def\Ed{$\mathsf{E2} $}
\def\Et{$\mathsf{E3} $}

\def\Ecq{\sf{(E5)} }

\def\Eo{\textsf{(E0)}}
\def\E15{\textsf{(E1-5)}}
\def\Eu{\textsf{(E1)}}
\def\Ed{\textsf{(E2)}}
\def\Et{\textsf{(E3)}}
\def\AP{\textsf{(AP)}}
\def\Eq{\textsf{(E4)}}
\def\EP{\textsf{(EP)}}
\def\HP{\textsf{(HP)}}
\def\WHP{\textsf{(WHP)}}
\def\diser{$(*)$}

\def\FIP{\textsf{FIP}}

\def\UP{\textsf{(UP)}}

\makeatother

\usepackage{babel}
  \addto\captionsafrikaans{}
  \addto\captionsafrikaans{}
  \addto\captionsafrikaans{}
  \addto\captionsafrikaans{}
  \addto\captionsafrikaans{}
  \addto\captionsenglish{}
  \addto\captionsenglish{}
  \addto\captionsenglish{}
  \addto\captionsenglish{}
  \addto\captionsenglish{}
  \addto\captionsitalian{}
  \addto\captionsitalian{}
  \addto\captionsitalian{}
  \addto\captionsitalian{}
  \addto\captionsitalian{}

\begin{document}

\title{A Euclidean comparison theory for the size of sets}

\author{Marco Forti}

\maketitle

\begin{abstract}\label{abs}
We discuss two main ways in comparing and evaluating the size of sets: the ``Cantorian" way, grounded on the so called Hume principle (two sets have equal size if they are equipotent), and the ``Euclidean" way, maintaining  Euclid's principle ``the whole is greater than the part".

The former being deeply investigated since the very birth of set theory, we  concentrate here on the ``Euclidean" notion of  size (\emph{numerosity})
 that maintains the Cantorain defiitions of \emph{order, addition and multiplication}, while preserving the natural 
  idea
that \emph{a set is (strictly) larger than its proper subsets}.
    These numerosities 
    satisfy the five Euclid's
\emph{common notions}, and constitute a
\emph{semiring
of nonstandarda natural numbers}, thus enjoying the best arithmetic.

Most relevant is the 
    \emph{natural set theoretic definition} of the set-preordering:\\
    ${}~~~~~~~~~~~~~~~~~~~~~~~~~~~~~~~~~~~~~~~~~~~~X\< Y\ \ \Iff\ \ \exists Z\ X\eq Z\pincl Y~~~~~~~~~~~~~~~~~~~~~~~~~~~~~~~~~~~~~~~~$
     Extending this ``proper subset property" from countable to uncountable sets 
    has been the main open question in this area from the beginning of the century.

\end{abstract}

\section*{Introduction}\label{intr}
In the history of Mathematics the problem of comparing (and measuring) the size of mathematical 
objects has been extensively studied. In particular, different methods have been experienced for associating to sets suitable  kinds 
of numbers.

 A satisfactory notion of \emph{measure of size} 
should abide by the famous five \emph{common notions} 
of Euclid's Elements, which traditionally 
embody the properties of any kind of  \emph{magnitudines} (see \cite{Eu}):
\begin{enumerate}
\item[\Eu]
\emph{Things equal to the same thing are also equal to one another.}
\item[\Ed]
\emph{And if equals be added to equals, the wholes are equal.}
\item[\Et]
\emph{And if equals be subtracted from equals, the remainders are equal.}
\item[\Eq]
\emph{Things \emph{[exactly] applying onto} one  another are  equal to one
another.}\footnote
{~Here we translate $\epsilon\phi\ag\rho\mu o\zeta o\nu\tau\!\ag$
by ``[exactly] applying onto", instead of the usual ``coinciding with". 
As pointed out by T.L.~Heath in his commentary \cite{Eu}, 
this translation seems to give a more appropriate
rendering of the mathematical usage of the verb
$\epsilon\phi\ag\rho\mu o\zeta \epsilon\iota\nu$.}
\item[\Ecq]
\emph{The whole is greater than the part.}
\end{enumerate}

Following the ancient praxis of \emph{comparing} magnitudes
of \emph{homogeneous objects},
 a very general notion of size of 
sets,\footnote{~so as to
comprehend cardinality, measure, probability, numerosity, \emph{etc}.}
whose essential property is  \emph{general comparability of sizes},
can be given through a total preordering\footnote{~Recall that a 
preordering is a \emph{reflexive} and \emph{transitive} (binary) 
relation; it is \emph{total} if any two elements are comparable; 
the corresponding 
 \emph{equivalence} $\simeq$ is $A\simeq B\ \ifg\ A\preceq B\preceq A$; the corresponding \emph{strict inequality} $\prec$ is $A\prec B\ \ifg\ A\preceq B\not\preceq A$ .}
$\preceq$ of sets according to their sizes,  with the intended meaning that \emph{equinumerosity, \ie\ equality of size}, is the corresponding equivalence relation $\ A\eq B\ \Iff\ A\ll B\ll A$, so
the first Euclidean common notion

\smallskip\noindent
\Eu\ ~~~~~ \ \emph{Things equal to the same thing are also equal to one another}

\smallskip\noindent
 is subsumed, because $\eq$ is an equivalence.

This comparison
should  naturally extend \emph{set-theoretic  inclusion}
 (and be \emph{consistent with equinumerosity}). So one has to assume that
$$A\ll B\ \ \Iff\ \  \exists A', B'  \big(A\eq A'\incl B'\eq B\big),$$

Clearly. the \emph{equivalence classes}  modulo $\eq$,~ called the \emph{magnitudines} of the theory,
are \emph{totally ordered} by 
the ordering induced by $\ll$. 

\medskip
In set theory the usual measure for the size of sets is  
 is given by the classical Cantorian
 notion of \emph{``cardinality''},  whose ground is the so 
 called \emph{Hume's Principle}
 \begin{center}
      \emph{Two sets 
  have the same size if and only if there exists a biunique 
  correspondence between  them.}
 \end{center}
This principle amounts to encompass the \emph{largest} possible class of ``exact" applications (\emph{congruences}) admissible in the fourth Euclidean notion, namely \emph{all bijections}. This 
assumption might seem natural, and even \emph{implicit in the notion of 
counting}; 
but it strongly violates the equally natural \emph{Euclid's principle} applied to 
(infinite) sets
 \begin{center}  
    \emph{A set 
     is greater than its proper subsets,}
 \end{center}
which in turn seems implicit in the notion of  \emph{magnitudo}, even for sets.

So one could distinguish  two basic kinds of \emph{size theories for sets}:
\begin{itemize}
\item A  size theory  is  \emph{Cantorian}
if, for all $A,B$:

    $(\mathsf{{HP}})$  
    ${(Hume's Principle)~~~}\ A\ll B\ \ \Iff\ \ \exists f:A\to B$ $1$-to-$1$
    
    (Cantor-Bernstein's theorem making this an  equivalent formulation)
\item A  size theory is  \emph{Euclidean}
if, for all $A,B$: 

   $(\mathsf{{EP}})$ 
    ${(Euclide's Principle)}\ \   \ A\< B \ \Iff\ \   \exists A', B'  \big(A\eq A'\pincl B'\eq B\big).$

(Remark the use of \emph{proper inclusion} in defining \emph{strict comparison} of sets)
\end{itemize}

 The spectacular development of Cantorian set theory 
in the entire twentieth century has put Euclid's principle in 
oblivion. Only the new millennium has seen a limited resurgence of 
proposals of so called ``numerosities" including it, at the cost of severe limitations of Hume's principle
(see  the excellent survey \cite{mancu} and the references therein). 

The main value of the Euclidean theories is the excellent arithmetic they allow, namely that of an \emph{ordered semiring},  to be contrasted with the awkward cardinal arithmetic.
However, the main problem arising in the Euclidean theories lies in the fact that the
\emph{preordering} of sets,  defined by \emph{the natural set theoretic characterization}
 \EP, should  induce the ``algebraic"   \emph{total} ordering of 
 the semiring of numerosities,
   which in turn is equivalent to
 the  \emph{subtraction} property   
$$ {\Diff}~~~~~~~~~~A\prec B\ \Iff\
  \exists C\ne\0\ \big(   A\cap C=\0\ \&\ A\cup C\eq B\big).~~~~~~~~$$
A notion of ``number of elements'' (\emph{numerosity}) that completely fulfills
 the Euclidean principle \EP\
 has been found up to now only for special \emph{countable }sets,  in \cite{BDNlab},  and generalized later to point sets on \emph{countable lines}  in \cite{QSU, FM}.  
The consistency of the full principle \EP\ for  \emph{uncountable sets} appeared problematic from the beginning, and this question has been posed in several papers (see \cite{BDNFar,BDNFuniv,DNFtup}), where only
the \emph{literal set-theoretic translation} of
      the fifth Euclidean notion, \ie\ the weaker principle requiring the sole left pointing arrow of $ \EP$,  
    $$ \Ecq{~~~~~~~~~~~~~~~~~~~~~~~~~~  ~}A\pincl B\ \ \Imp\ \ A\< B{~~~~~~~~~~~~~~~~~~~~~~~~~},$$
      has been obtained.
       (On the other hand, it is worth recalling that also the totality of the Cantorian weak cardinal ordering 
had to wait a couple of decades before \emph{Zermelo's new axiom of choice} established it!)

In this paper we present 
 a Euclidean numerosity theory for suitable collections $\WW$ of point sets of  finite dimensional spaces over lines  $\LL$ of arbitrary cardinality,  satisfying the \emph{full principle} \EP; this theory  might be extended to the whole universe of sets $V$ following the procedure outlined in \cite{BDNFar}, under 
mild  set theoretic assumptions, \pes\ Von Neumann's axiom, that gives a (class-)bijection between the universe $V$ and the class $Ord$ of all ordinals.

\section{Euclidean or\! Aristotelian\! comparison\! theories}\label{ct}
As pointed out in the introduction, in \emph{comparing} magnitudes
of \emph{homogeneous objects},
the essential property is the \emph{general comparability of sizes}, 
but a Euclidean comparison
 $\prec$  among
sets 
should  naturally extend \emph{set theoretic  inclusion $\pincl$}, according to the fifth Euclidean common notion,
 (and be \emph{consistent with equinumerosity}). 
\begin{defn}\label{compar}
Call \emph{Euclidean comparison theory} a pair $( \mathbb{W},\preceq) $
where
\begin{enumerate}
\item $\mathbb{W}$ is
 a family of sets 
 closed under  
%
\emph{binary unions and intersections},
 \emph{subsets}, 
 and also
under \emph{ Cartesian products};
\item $\,\preceq\,$ is a \emph{total preordering}
such that, for all $A, B\in\mathbb{W},$
$${\EP} ~~~~~~~~~~~~~~~~~~~~~~A\prec B\ \ \Iff\ \, \exists B'\in\WW\ A\subset B'\eq   B~~~~~~~~~~~~~~~~~~~~~~~~~$$
 \item The  \emph{universe} of the theory
$\left( \mathbb{W},\preceq \right) $ is the union set
$W=\bigcup \mathbb{W},$
and
a permutation of the universe $\sg\in\Sk(W)$\footnote{~ $\Sk (X)$ denotes the group of all permutations of a set $X$.} is a \emph{congruence} for 
$(\WW,\preceq)$) if 

$\ {}~~~~~~~~~~~~~~~~~~~~~~~for \ all\ A\in \mathbb{W}\ \ \sg(A)\in\WW \ {and} \ 
\sg(A)\simeq A.~~~~~~~~~~~~~~~~~~~~~
$
\end{enumerate}
The \emph{quotient set} $\ \Nk=\mathbb{W}/\eq\,$ is the \emph{set of numerosities} of the theory, and the canonical map $\ \nk:\WW\to\Nk$ is the \emph{numerosity function} of the theory. 
\end{defn}
Clearly $\Nk$
is \emph{totally ordered} by 
the ordering induced by $\ll$. 

 \subsection{Natural congruences}\label{cong}
First of all, once the general Hume's principle cannot be assumed,
 the fourth 
Euclid's common notion
\begin{center}
\Eq\  ~~~  \emph{Things exactly applying  onto one 
     another are equal to one another}~~~~~~~
\end{center}
is left in need of an adequate interpretation that
 identifies an appropriate  class of natural \emph{exact applications that preserve size} (called \emph{congruences}).
So we have to isolate a   \emph{family of congruences}, for the considered notion of size,
 as a subset $\Ck(W)$ of
the group of all permutations $\Sk (W)$ of the universe $W$.

  Call
 ``natural transformation'' 
 of tuples any biunique correspondence $\tau$ that \emph{preserves support},
 \ie\ the set of components of a tuple:
    $supp(a_1,\ldots,a_n)=\{a_1,\ldots,a_n\}$): 
 these applications are useful when 
 comparing sets of different dimensions, 
so they  seem a good basis to be put in $\Gk(W)$ by any Euclidean theory involving sets of tuples.\footnote{
~These transformations may \emph{not preserve dimension}. When dimension is relevant, 
\pes\ when the diagonal $D_A=\{(a,a)\mid a\in A\}$ shoud require a different size from $A$, one could restrict consideration to permutation of components and/or rearranging of parentheses, ~\ie\  $\tau:(a_1,\ldots,a_n)\mapsto[a_{\sg1},\ldots,a_{\sg n}]$, where $[\ldots]$ represents any distribution of parentheses.}
So we assume that all natural transformation of tuples belong to $\Gk(W),$ and
  postulate
 $${\CP}\, ({Congruence\ Principle})\ \ 
   \tau\in\Ck(W)\  \Imp\  \forall A\in \mathbb{W}\,  \big(\tau[A]\in\WW \ \mathrm{\&} \ \tau[A]\simeq A\big)$$

\subsection{Addition of numerosities}\label{add}
In general one wants not only \emph{compare}, but also \emph{add and subtract
 magnitudines}, according to the second and third Euclidean common notions

\noindent 
\Ed\ \emph{{~~~~~~~~~}... if equals be added to equals, the wholes are equal.{~~~~~~~~~~~~~~~~}}

\noindent \Et\ ~~\emph{...  if equals be subtracted from equals, the remainders are equal.}

\medskip
When dealing with sets, it is natural to take \emph{addition} to be \emph{(disjoint) union}, and \emph{subtraction} to be \emph{(relative) complement}, so
\smallskip  \noindent  
 it is convenient to call \emph{additive} a Euclidean comparison theory verifying the following principle
 for all $A,B\in\WW$:
\begin{description}
\item[$(\mathsf{{AP}})$]  (\emph{Aristotle's Principle})\footnote{~
This priciple has been named \emph{Aristotle's Principle} in \cite{QSU,FM}, 
because it resembles     Aristotle's preferred example of a ``general axiom". It is especially relevant in this context, because \AP\ implies both the second and the third 
     Euclidean common notions, and also the fifth 
    provided that no 
     nonempty set is equivalent to $\0$, see below. 
}
    ${~~~~~~~~~~~}
    A\eq B\ \Longleftrightarrow\
    A\7 B \eq B\7 A.
    $
\end{description}
This principle yields both the second and third Euclidean common notions, 
namely
\begin{prop}[\cite{DNFtup,QSU}]${}$\ Assuming \rm{$\AP$},
the following properties hold for all $A,B,A',B'\in\WW$:
\begin{description}
\item[$(\mathsf{{E2}})$] ${~~~~~}  A\simeq A^{\prime },\ B\simeq B^{\prime }, A\cap B=A'\cap B' =\0\ \ \Imp\ \ A\cup B\simeq A'\cup B'
$
\item[$(\mathsf{{E3}})$] ${~~~~~~~~} B\subset
A,\ B^{\prime }\subset A^{\prime },
  A\simeq A^{\prime },\ B\simeq B^{\prime }\Longrightarrow A\backslash B\simeq
A^{\prime }\backslash B^{\prime }.$
\end{description}
\end{prop}
We omit the proofs, which would be identical to those given in 
\cite{DNFtup,QSU}.
\qed

\bigskip
\begin{rem}
The sole principle \AP\ is not enaugh for providing a Euclidean comparison:
for instance, it is fulfilled both by Peano and Lebesgue 
measures, by taking\; $
\mathbb{W}$  to be a suitable subset of $\ \bigcup_{n\in\omega}\mathbb{R}^{n}
$. In general,  various probability spaces might be taken to be $\WW$, 
but these theories are non-Euclidean, unless the probability is \emph{regular} 
(\ie\ only the whole space has probability $1$).
 Actually, by simply adding to \AP\ the natural principle
 $$\Eo~~~~~~~~~~~~~~~~~~~~~~~~~~~~~~~\ \ A\eq\0\ \ \Imp\ \ A=\0,~~~~~~~~~~~~~~~~~~~~~~~~~~~~~~~~~~~~~$$
one obtains the literal set theoretic version of the fifth Euclidean notion, namely
 $${ \Ecq}{~~~~~~~~~~~~~~~~~~~~~~~~~~~~~~~~  ~}A\pincl B\ \ \Imp\ \ A\< B{~~~~~~~~~~~~~~~~~~~~~~~~~~~~~~~~~~~},$$

Clearly these assumptions yield that $\0$ is \emph{the unique least set}, and that 
 \emph{all singletons have equal size} and come immediately after $\0$,
  for no nonempty set
can be smaller than a singleton. Moreover
 the \emph{successor} of the size of a given set is obtained
 by adding a single element, and similarly the \emph{immediate predecessor} is obtained by removing one element,
so the induced total ordering is \emph{discrete}. Hence we may assume without loss of generality that 
all \emph{finite} sets receives \emph{their number of elements} as size
(thus the natural numbers $\N$ can be viewed as an \emph{initial segment} of all numerosities).
\end{rem}

Call \emph{Aristotelian} a size theory satisfying both principles \AP\ and \Eo:\footnote{~ To be sure, Aristotle had never accepted a theory where $1=2$!}
the important feature of Aristotelian theories is that \AP\ and \Eo\ are \emph{necessary and sufficient conditions}  for obtaining \emph{an excellent additive arithmetic of the infinite}: simply define
an \emph{addition  of numerosities}, \ie\  equivalence classes modulo $\eq$,
 by means of disjoint union
 $$\nk(A)+\nk(B)=\nk(A\cup B)\ \ for\ all \ \ A,B\in\WW \  such\ that\ \ A\cap B=\0,$$
  then

\begin{thm}\label{mon}
Let $\left( \mathbb{W},\preceq \right) $ be an Aristotelian comparison theory. Then
an additon can be defined on
 the 
corresponding
 set of numerositiess $\Nk=\mathbb{W}/\eq\,$~ in such a way that
 $$\nk(A)+\nk(B)=\nk(A\cup B)+\nk(A\cap B)\ \ for\ all \ \ A,B\in\WW,$$
 $$ \exists C\not\incl A \ ( \nk(A)+\nk(C)=\nk(B))\ \ \Imp\ \ \nk(A)<\nk(B),$$
and then 
$(\Nk,+,0,\le)$ is a commutative cancellative zerosumfree  monoid
whose algebraic  ordering\footnote{~
  Recall that any cancellative zerosumfree monoid comes
with the \emph{algebraic} partial ordering $\le$ defined
by letting\ 
$\ a\le b\ \Longleftrightarrow  \exists c.\,(\,a+c=b\,);$
this ordering is total if and only if the monoid is  \emph{semisubtractive} \ie\ for all $a,b$ there exists $c$ such that either $a+c=b$ or $b+c=a$.
Hence the preordering $\ll$ on sets induces the algebraic ordering  of $\Nk$ if and only if the latter is total.} 
$\le$ is weaker than that induced by the total preordering $\ll$.
Therefore $\Nk$ can be isomorphically embedded into the nonnegative part 
$\,\Ak^{\le 0}$ of an ordered abelian group $\,\Ak$.

The monoid  $\Nk$ is semisubtractive
if and only if $\left( \mathbb{W},\preceq \right) $ is Euclidean, and then 
$$ \exists C\not\incl A \ ( \nk(A)+\nk(C)=\nk(B))\ \ \Iff\ \ \nk(A)<\nk(B),$$
so $\ll$ induces the algebraic ordering of $\Nk$, and
$\,\Nk\cong\Ak^{\le 0}$.
\end{thm} 

\begin{proof}
First of all,
$(\Nk,+,0,\le)$ is a a \emph{commutative cancellative} \emph{monoid}, because
 addition 
is \emph{associative} and \emph{commutative} by definition, and the property
  \Et, that follows from \AP, provides the \emph{cancellation property}
 $$\nk(A)+\nk(C)=\nk(B)+\nk(C)\ \ \Imp\ \ \nk(A)=\nk(B).$$
  $\Nk$ is \emph{zerosumfree} because
 $0=\nk(\0)$ is the unique additively neutral element, and so
 $$\ \nk(A)+\nk(B)=0=\nk(\emptyset)\ \ \Imp\  \ A=B=\emptyset.$$
 
 Finally \EP\ is equivalent to semisubtractivity, because 
 given $A,B\in\WW$ there exists $C\in\WW$, disjoint from both $A$ and $B$, such that either $A\cup C\eq B$ or $B\cup C\eq A$, and so the
algebraic  ordering  of $\Nk$ coincides with that induced by the total preordering  $\ll$.

 The monoid $\Nk$ generates an abelian group $\Ak$, whose elements can be identified with the equivalence classes of the \emph{differences} $\nk(A)-\nk(B)$ modulo the equivalence 
$\ \ \ \nk(A)-\nk(B)\approx\nk(A')-\nk(B')\  \Iff\ \ \nk(A')+\nk(B)=\nk(A)+\nk(B').$
Clearly one has 
$\  \nk(A)-\nk(B)> 0$ in $\Ak$ if  $\exists C\ne\0\ (\nk(A)=\nk(B)+\nk(C)$, and the reverse implication holds if and only if $\,\Nk$ is semisubtractive.

\qed \end{proof}

The fact that the preordering $\<$  on sets induces the  \emph{algebraic total ordering on }$\Nk,$
yields the  \emph{``most wanted Subtraction Principle"} of \cite{BDNFar}
$$\ {\Diff}~~~~~~~~~~    \emph{ $A\< B\ \ \Iff\ \ \exists C\ne\0  \ \big(C\cap (A\cup B)=\0,
      \  
    (C\cup A)\eq B\big)$.~~~~~~~~~~~~~~~~~~}$$
Clearly the equivalence class  of the set $C$ is uniquely determined,  and in the group $\Ak$  every element has the form
$\pm\nk(C)$ with $\nk(C)\in\Nk$.

The  problem of the relative consistency of  \EP\  with \AP\ 
(thus yielding the Subtraction Principle) has been posed in several papers dealing with 
 \emph{Aristotelian} notions of size for sets (see \pes\  \cite{BDNFar,BDNFuniv}). However a positive answer has been obtained, up to now, only for \emph{countable} sets in \cite{DNFtup,QSU,FM}, thanks to the consistency of \emph{selective }or \emph{quasiselective ultrafilters on} $\N$. The  consistency of the Subtraction Principle with both {\Ecq}and \AP\
 for sets of arbitrary cardinality (which is equivalent to the existence of Euclidean ultrafilters, as proved in \cite{DNFeu})
 is in fact the main result of this paper, but leaves open the consistency problem for the conjuction of full \EP\ and \AP. 

\medskip
\begin{rem}
Let us call \emph{weakly additive} a comparison theory satisfying the condition \Ed, but not \Et.
It is well known that any Cantorian theory is weakly additive, in fact 
 the 
corresponding
set of magnitudines $\Mk=\mathbb{\WW}/\eq\,$~
can be identified with a set of \emph{cardinal numbers}. So the 
corresponding sum of infinite numbers is trivialized by the 
general equality
$$\ak +\bk= \max\{\ak,\bk\},$$
and  cannot admit an inverse operation. Actually, the very failure of the Euclid's principle has been taken as \emph{definition of infinity} by Dedekind.
\end{rem}

    
\subsection{Multiplication of numerosities}\label{mult}
In classical mathematics, only \emph{homogeneous} magnitudes are \emph{comparable}, and geometric figures having different dimensions are never compared,  so a \emph{multiplicative} version of Euclid's second common notion
\begin{center}
   \emph{...  if 
        equals be multiplied by equals, the products are equal}
\end{center}
was not considered, in the presence of different ``dimensions''.  

On the other hand,  in modern mathematics a single
set of ``numbers", the real numbers $\R$,
is used as a common scale for
 measuring the size of figures of any dimension.
In a general set theoretic context it seems natural to consider abstract sets as homogeneous mathematical objects, without distinctions based on dimension,
and Georg Cantor introduced his theory of \emph{cardinal numbers} in order to give a measure of size of arbitrary sets. In the same vein, we introduce the notion of \emph{numerosity}, aiming to provide a general Euclidean measure of size for sets, more adherent to the classical conception of \emph{magnitudo}.

A satisfying \emph{arithmetic} of numerosities needs a \emph{product} (and a corresponding \emph{unit}), and
we adhere to the natural Cantorian choice
 of introducing a product 
through \emph{Cartesian products},\footnote{~
 \emph{CAVEAT}: the Cartesian product is optimal when any two sets  $A, B$ are \emph{multipliable} in the sense that their Cartesian product is disjoint from their union, 
but when \emph{entire transitive universes} like $V_\kg, H(\kg)$, or $L$ are considered, the usual set-theoretic coding of pairs makes it untenable (\pes\ already $V_\og\*\{x\}\pincl V_\og$ for any $x\in V_\og$), hence, 
as already done for addition, in these cases 
one should assume the existence of suitable \emph{multipliable copies} of any set.
}
 and  taking \emph{singletons as unitary}.
Although the Cartesian product is neither commutative nor associative \emph{stricto sensu}, nevertheless the corresponding natural transformations have been taken  among the congruences in the set  $\Ck(W)$, hence are numerosity preserving.

 Moreover, it seems  natural  that multiplication by ``suitable"\footnote{~
As remarked above, not all singletons may be ``suitable" for a Euclidean theory.}
 singletons
 be such a ``congruence'' to be put into  $\Gk(W)$, so as to have at disposal disjoint equinumerous copies to be used in summing and multiplying numerosities.
In doing so, each product $A\*\{b\},\ b\in B$ may be viewed as a \emph{disjoint equinumerous copy of }$A$, thus  making 
  their (disjoint) union $A\*B$  the sum of ``$B$-many copies of $A$", in accord with the arithmetic interpretation of multiplication.

 \subsection{ Aristotelian and Euclidean numerosity theories}\label{num}

The preceding discussion leads to the following definition
\begin{defn}
An Aristotelian  comparison theory 
$\,( \mathbb{W},\preceq) $ is  a \emph {numerosity}\footnote{~
in the sequel, for sake of brevity, we often omit the specification `Aristotelian': actually, the principles  \PP\ and \UP\ together imply \Eo, otherwise all sets are null.}  theory if,
for all $A,B,C\in\WW,$ with $C\ne\0$, and all $w\in W$ 
\begin{itemize}
\item[\PP] 
$((A\cup B)\* C) \,\cap (A\cup B\cup C) =\0 \ \ \Imp\
\big(A\ll B \ \ \Iff \ \ A\* C\ll B\* C\big)$;  
\item[\UP]  $\ A\eq (A\* \{w\})$ for all $w$ such that $(A\*\{w\})\cap A=\0$.
\end{itemize}
The numerosity is \emph{Euclidean} if the full principle \EP\ holds.
\end{defn}
 The above  Principles provide the
 set of numerosities $\Nk=\mathbb{W}/\eq$
with the best algebraic properties, namely
\begin{thm}\label{eunum}
Let $\left\{ \mathbb{W},\preceq\right\} $ be a numerosity. Then  the set of numerosities
$\Nk$ has a natural structure of ordered semiring,
where addition corresponds to disjoint union and multiplication to Cartesian 
product. 
Therefore there exist an ordered ring $\Ak$ and an embedding
$
\nk:\,\mathbb{W}\rightarrow\Ak^{\geq 0}$
such that, for all $A,B\in\WW$,\

\smallskip\noindent
$\nk(A)+\nk(B)=\nk(A\cup B)+\nk(A\cap B),\
\nk(A\times B)=\nk(A)\cdot\nk(B),\
A\!\pincl\! B \imp \nk(A)\!\!<\!\nk(B)
$

\smallskip
In particular the five Euclid's Common Notions \E15\ are satisfied, and all finite sets receive their number of
elements as numerosity,  so $\Nk$ contains as initial segment an isomorphic copy of 
the natural numbers $\N$.

\end{thm}\label{eucl}
\begin{proof} 
    The proof is close to that of Theorem 4.2 in \cite{QSU}:
    here axiom {\CP}\ makes multiplication commutative and associative,
     Principles \AP, \PP\ and \UP\  are the Axioms \Eu,\Eq\ and \Et\ of \cite{QSU}, 
   and together imply
     \Eo\  (unless the numerosity is trivially $0$), hence also {\Ecq}, 
    while \Ed\ of \cite{QSU}\footnote{~
    Actually \Ed\ of \cite{QSU} is equivalent to the proper subset property of Subsection \ref{sub}, which is even stronger than \EP.
    }  can be replaced here by totality of the preordering $\ll$.

 We have already proved  Theorem \ref{mon} that $(\Nk,+,0,\le)$ is a commutative cancellative zerosumfree semisubtractive monoid, with the algebraic ordering $\le$ possibly weaker than that induced by the set-theoretic preordering $\ll$.
 Now also $(\Nk,\cdot,1)$ is a commutative and associative monoid, that is distributive w.r.t.~$+$, annihilated by $0$,  cancellative (because the Cartesian product of nonempty sets is nonempty), and has multiplicative unit  $1\!=\nk(\{w\})$ by \UP.
Finally, the abelian group $\Ak$ defined in the proof of Theorem \ref{mon} becomes an ordered ring by simply defining $\nk(A)\cdot\nk(B)=\nk(A\times B)$.

Clearly  there exists a \emph{unique} numerosity 
where $\WW$ is the family of all \emph{finite} sets, namely that given by 
the \emph{number of elements}, so the corresponding set of numerosities 
$\Nk$ is an isomorphic copy of 
the natural numbers $\N$, which is an initial segment of every semiring of numerosities.
   \qed
\end{proof}

 \section{Numerosity theories for ``Punktmengen"}\label{numpt}
 
In  this section we specify  the notion of  {(Aristotelian \rm{or} Euclidean) numerosity} for ``Punkt\-mengen" (finitary point-sets) over a ``line" $\LL$, \ie\ subsets $A\incl\bigcup_{n\in\N} \LL^n)$ \st\  $\{a\in A \mid supp(a)=i\}$ \textrm{ is}\ 
finite for all $i\in\LL^{<\og}$.  Actually, in a general set-theoretic context, there are no ``geometric" or ``analytic" properties to be considered: the sole relevant characteristic of the line $\LL$ remains  \emph{cardinality}, so a possible choice seems to be simply identifying $\LL$ with its cardinal $\kg$, thus obtaining the fringe benefit that no pair of ordinals is an ordinal, and Cartesian products may be freely used.\footnote{~
In any case it seems convenient to assume that the line $\LL$ is disjoint from its square $\LL^2$.}

Grounding on the preceding discussion, we  pose the following definition
\begin{defn}\label{Eunum}${}$\
 $(\WW,\ll)$ is a 
 \emph{numerosity theory} for ``Punktmengen" over $\LL$ if
 \begin{itemize}
  \item $\WW\incl \P(\bigcup_{n\in\N} \LL^n)$ is the collection of all finitary point sets over the line $\LL$;
  \item $\ll$ is a total preordering on $\WW$,  and $\eq$ the equivalence generated by $\ll$;
 \item   the following conditions are satisfied for all $A,B,C\in\WW$:
\begin{itemize}
\item[$(\mathsf{{AP}})$]$A\eq B\ \Longleftrightarrow\ A\7 B \eq B\7 A;$
  \item[\PP] 
 $A\eq B \ \ \Iff \ \ A\* C\eq B\* C$ (for all $C\ne\0$); 
\item[\UP] $A\eq A\* \{w\}$ for all $w\in W=\bigcup\WW$;
\item[\CP]  $\tau[A]\simeq A$ for all $\tau \in\Ck(\LL).$\footnote{~ Recall that 
$\Ck(\LL)$ is the set of all \emph{support preserving bijections}, see Subsect.\ref{cong}}
\end{itemize}
\end{itemize}
\noindent
The numerosity $(\WW,\ll)$ is \emph{Euclidean}\footnote{~
Remark that \PP\ and \UP\ together imply \Eo, so $(\WW,\ll)$ is always Aristotelian.} if satisfies the Euclidean principle 

\smallskip
$~~(\mathsf{{EP}})\ 
 A\prec B\ \Iff\ \exists B'(A\subset B'\eq B),$
with  \emph{strict inclusion} and \emph{preordering}.
     \end{defn}

The idea that \emph{global} properties of sets might be tested 
\emph{``locally''} 
suggests the following 
  \begin{defn}
    Let $(\mathbb{W},\preceq) $ be a  numerosity for ``Punktmengen" over $\LL$,
 let $ \II=[\LL]^{<\og}$ be the set of all finite subsets of $\LL$,
 and for $\ A\in\WW$ and $ \ i\in\II$, let $A_i=\{a\in A\!\mid\! supp(a)\incl i\}.$
 \begin{itemize}
  \item  the  \emph{counting function} $ \Phi:\WW\to\N^\II$    is given by  
   $\Phi(A)=\la |A_i|\mid i\in\II\ra$. 
  \item  $(\mathbb{W},\preceq) $ is \emph{finitely approximable} if\ \ 
  $~
   \forall i\in \II\,\,(|A_i|\le |B_i|)
    \ \Imp\  A\ll B.$
\end{itemize} 
\end{defn}

 The following lemma shows the ``algebraic character" of finitary approximable numerosities.    
       \begin{lem}\label{cf}
The set of all differences $\Phi(A)-\Phi(B)$, with $A,B\in\WW$ covers the whole ring $\Z^{\II}$.
\end{lem}
\proof\ 
 Wellorder $\LL$ in type $\kg$ and then the set $\II$ of all finite subsets of $\LL$ according to $\ i<j\ \ifg\ \max(i\Delta j\in j)$ (hence again in type $\kg$): this ordering is \st\ all proper subsets of  any $i_\ag\in\II$ appear at stages $\bg<\ag$.
 
Given $\zk\in\Z^\II$,  define  inductively on $\ag$ increasing sequences of  finite sets $A^{(\ag)}\incl A,\ B^{(\ag)}\incl B$  whose support is exactly $i_\ag$,  in such a way that 
$|{A_i}|-|{B_i}|$ be the $i^{th}$ component of $\zk$.
Assuming this for all $\bg<\ag$,
   pick a  number of tuples whose support is exactly $i_\ag$ and  put them either in $A^{(\ag)}$ or in $B^{(\ag)}$, so as to adjust the value of 
  $|A_{i_\ag}|-|B_{i_\ag}|$ to be $z_{i_\ag}$. 
 These new elements cannot have been taken before, because the elements of any 
$A_j, B_j$ with $j$ preceding $i$ cannot have support $\linc i$, hence the $\ag^{th}$ step can be done, because of the inclusion-exclusion principle: 
$|A_{i}|=\sum_{j\pincl i}(-1)^{|i\7 j|}
\left( \begin{array}{c} |i| \\ |j|  \end{array} \right)|A_{j}|$:
\qed

\bigskip
Following \cite{EDM}, call \emph{Euclidean} a fine ultrafilters\footnote{~
   Recall that a filter on $\II$ is fine if it contains all ``cones" $C(j)=\{i\in\II\mid j\incl i\}.$}
 $\,\U$ on $\II$ if 
 $$\forall \psi\in\N^\II \,\exists U_\psi\in\U\,
  \forall i,j\in\U_\psi\ \big(i\pincl j\ \Imp\ \psi(i)\le\psi(j)\big).$$
 \noindent
Finite approximability  allows for a ``concrete" 
 strengthening of Theorem \ref{eucl}, 
leading naturally to \emph{hypernatural numbers as numerosities},
namely
 \begin{thm}\label{ult}
 There is a biunique correspondence between finitely approximable numerosities 
   $(\mathbb{W},\preceq) $  and fine ultrafilters
 $\,\U$ on $\II$.
 If $\,\U$ corresponds to  $(\mathbb{W},\preceq) $  in this correspondence, then 
$$ (\ref{ult})~~~~~~~~~~~\forall A,B\in\WW\!\ \big( A\ll B\ \Iff\ 
 \{i\in\II\mid |A_i|\le|B_i|\}\in\U
 \big),~~~~~~~~~~~~$$
and  there is
 an 
 isomorphic embedding  $\fg$ of  the semiring\! of\! numerosities $\Nk\!=\!\WW/\!\!\eq$ \ into the ultrapower $N\ult{\II}{\U}$ that makes the following diagram commute:
 
 \bigskip
\begin{center}
\begin{picture}(90,70)
   \put(0,0){\makebox(0,0){$\Nk$}}
   \put(101,0){\makebox(0,0){$\N\ult{\II}{\U}\!\pincl\Z\ult{\II}{\U}$}}
   \put(0,70){\makebox(0,0){$\WW$}}
   \put(45,35){\makebox(0,0){\diser}}
   \put(99,70){\makebox(0,0){$\N^\II\pincl \Z^\II$}}
   \put(45,6){\makebox(0,0){$\fg$}}
   \put(45,76){\makebox(0,0){$\Phi$}}
   \put(-8,35){\makebox(0,0){$\nk$}}
   \put(98,35){\makebox(0,0){$\pi_\U$}}
   \put(15,0){\vector(1,0){60}}
   \put(15,70){\vector(1,0){60}}
   \put(110,60){\vector(0,-1){50}}
    \put(0,60){\vector(0,-1){50}}
   \put(86,60){\vector(0,-1){50}}
\end{picture}
\end{center}

\begin{center}
    $($where $\Phi$ maps any $A\in\WW$ to its \emph{counting function} 
   $ \Phi({A}):\II\to\N)$

\end{center}
Moreover the set of congruence can be taken to be 
$$\Ck_\U(\WW)=\{\tau \mid \exists U\in\U\,\forall i\in U\,\tau[i]=i\}.$$
Finally the numerosity is Euclidean if and only if the ultrafilter is Euclidean.
 \end{thm}
 
 \proof\  Given a fine ultrafilter $\,\U$ on $\II$, define $\ll$ by $(\ref{ult})$.
 Then 
 \begin{itemize}
  \item 
 {\CP} holds because congruent sets share the same counting functions,
 \item  \UP\  holds taking $\ d=supp(w)$\ and $\ U=C(d)\in\U$.
  \item 
 \AP\ holds because $|(A\cup B)_i|=|A_i\cup B_i|=|A_i|+|B_i|$, if  $(A\cap B)=\0$,   and so
 $|(A\7 B)_i|+|(A\cap B)_i|=|A_i|$.
\item \PP\ holds because $supp(a,b)=supp(a)\cup supp(b)$, hence\\
  ${~~~~~~~~~~~}(A\* B)_i=\{(a,b)\mid supp(a)\cup supp(b)\incl i\}=A_i\* B_i$.
\end{itemize} 
  Clearly these equalities continue to hold by passing to the ultrapower modulo the fine ultrafilter $\,\U$, hence  $(\mathbb{W},\preceq) $ is finitely approximable, and the unique map $\fg:\nk(A)\mapsto\pi_\U(\Phi(A))$ is welldefined and preserves sums and products.
  Finally, the preordering $\ll$ is total in any case, because $\U$ is ultra.
  
  Moreover, if $\,\U$ is Euclidean and $A\ll B$, \ie\
   the difference  $\Phi(B)-\Phi(A)$  is nonnegative on a set 
   $U\in\U$, or equivalently  is equal on $U$ to some $\psi\in\Nai$, then $\pi_\U(\psi)\in\N\ult{\II}{\U}$ and there exists $C$ \st\ $\psi=\phi(\nk(C))$, equivalently $\nk(B)+\nk(C)=\nk(A)$, so {\Diff} holds.
  
\smallskip 
 Conversely, given $(\mathbb{W},\preceq) $  finitely approximable,  the family of sets
 $$\F=\big\{U_{AB}= \{i\in\II\mid |A_i|=|B_i|\}\mid A\eq B\big\}$$  has the \FIP, so it generates a filter, which is fine because 
 $C(\{w,v\})=U_{\{w\}\{v\}}$. 
 The ring $\Ak$ generated by $\Phi[\WW]$ is $\Zai$, by Lemma \ref{cf}, and the total preordering $\ll$ induces a total ordering on $\Z\ult{\II}{\U}$, by$( \ref{ult})$, hence $\,\U$ is ultra, and the map $\phi$ is injective and preserves sums, products, and ordering of $\Nk$.
%

Moreover  $\tau\in\Ck_\U(\WW)$ trivially implies 
$\{i\in\II\mid |\tau[A]_i|=|A_i|\}\in\U$.

  Finally \EP\ implies that the difference of two counting functions, when positive on a set in $\,U$, is equivalent  modulo $\,\U$ to some counting function, hence the  set $\Nk$ of all numerosities is mapped by $\fg$ onto the ultrapower $\N\ult{\II}{\U}$. 

  \qed
 
 \bigskip
Therefore finitely approximable numerosities exist for lines $\LL$ of arbitrary cardinality, and Euclidean approximable numerosities exist for $\LL$ if and only if  there are Euclidean ultrafilters on $\II$.
We are left with the question of the existence of Euclidean ultrafilters.

The paper  \cite{je}\ studies the partition property 
$\ [A]^{<\omega}\to(\omega,{cofin})^2_\pincl$ affirming that
 \emph{any $2$-partition  $G:[[A]^{<\omega}]]^2\to\{0,1\}$ of the pairs of 
finite subsets of $A$ has a $\pincl$-cofinal homogeneous set 
$H\incl A^{<\og}$.}
In \cite{js}  its validity for $|A|=\ahu$ is established, but  the problem for larger sets is left open.
 The recent  paper \cite{EDM}  introduces the similar ``unbalanced" property
  $$\ [A]^{<\omega}\to(\omega,{cofin})^2_\pincl$$  affirming that \emph{any $2$-partition  $G:[[A]^{<\omega}]]^2\to\{0,1\}$ of the pairs of 
finite subsets of $A$, either admits a $0$-chain (\ie\  a $\pincl$-increasing sequence $a_n\in[A]^{<\og}$ with $G(a_n,a_{n+1})=0)$, or it has a $\pincl$-cofinal homogeneous set 
$H\incl A^{<\og}$ (hence necessarily $G(a,b)=1$ for all $a,b\in H,\ a\pincl b).$} 

\bigskip
This partition property is all that is needed in order to have Euclidean ultrafilters, hence Euclidean numerosities, namely
\begin{lem} \label{ultEu}${}$ \emph{(see Lemma $1.4$ of \cite{DNFeu})}

 If $\ \II\to(\og,cofin)^2_\pincl $ holds, then there are Euclidean ultrafilters on $\II$.  
\end{lem}
\proof \ For $\psi\in\Nai$, define
 the partition
 $$G_{\psi}:[\II]^2_\pincl \to \{0,1\}\ \ \textrm{by}\ \ G_\psi(i,j)=\begin{cases}
  0    & \text{if } \psi(i)>\psi(j), \\
   1   & \text{otherwise}.
\end{cases}$$
Given finitely many $\psi_k\in\Nai$, let $\psi=\prod_k\psi_k$: then $\psi$
cannot admit $0$-chains,
  so there is a $\pincl$-cofinal   $1$-homogeneous set $H_\psi$, which is  
   simultaneously $1$-homogeneous for all $G_{\psi_k}$.
 Hence the family  $\Hl=\{H_\psi\mid \psi\in\Nai\}$ 
  has the \FIP,
and any fine  ultrafilter $\U$ on $\II$ including $\Hl$ is Euclidean.
\qed

\bigskip
Actually, the partition property $\ [A]^{<\omega}\to(\omega,\text{cofinal})^2_\pincl$ has been  stated for all sets $A$ of any cardinality $\kg$ as the Main Theorem of \cite{EDM}, 
so
the existence of Euclidean ultrafilters on $\II$, hence of Euclidean numerosities  satisfying the subtraction property {\Diff} is granted on the family $\WW$
of all (finitary) point sets over lines $\LL$ of arbitrary cardinality.

\section{The Weak Hume's Principle}\label{hum}
We remark that the Cantorian theory of cardinality and the Euclidean theory of numerosity might be reconciled by \emph{weakening} the first one, while \emph{slightly strengthening} the latter.
In fact, on the one hand, the Cantorian theory uses this  form of Hume's Principle 
$$\HP{~~~~~~~~~~~~~~~~~~~~~~~} A\ll B\ \ \ \Iff\ \ \ \exists f: A\to B,\, f\, 1\mbox{- to -}1.{~~~~~~~~~~~~~~~~~~~~~~~~~~}$$
instead  of Euclid's principle \EP\ 
in defining the (weak) preordering $\ll$ of sets.\\
 Then the  principles 
$\Ed$ and {\PP,\UP,\CP} 
 become provable, while $\Et$ and $\Ecq$ (hence also \AP\ and \EP) are refuted.

On the other hand,  perhaps the best way to view Aristotelian or Euclidean numerosities is looking at them as a \emph{refinement} of Cantorian cardinality, able to separate sets that, although \emph{equipotent}, have in fact \emph{really different sizes}, in particular when they are \emph{proper} subsets or supersets of one another. This  conception 
amounts in
 ``weakly Cantorianizing" the numerosity theory by  adding the sole ``only if\," part of Hume's principle:\footnote{~This property is called ``Half Cantor Principle" in \cite{BDNFar}.}
 \begin{center}
  \ \ \emph{If two sets 
 are equinumerous, then there exists a biunique 
  correspondence between  them.}
\end{center}

\subsection{Weakly Humean numerosity}\label{whum}
The idea of getting a  ``weakly Cantorian"  numerosity theory is realized
by  simply adding   the right pointing arrow of \HP\
 {(named Half Cantor Principle in \cite{BDNFar}) 
 to the principle \AP\  namely
\begin{defn}
An (Aristotelian or Euclidean) numerosity theory is\emph{ weakly Humean} if
it satisfies  the following
$${\WHP}\ \ (Weak\ Hume's\ Principle)~~~~~A\preceq B\ \ \Imp\ \  \exists 
  f\, 1\mbox{- to -}1,\, f: A\to B.{~~~~}$$
\end{defn}
If one wants to maintain also the finite approximability, it amounts to require that the fine ultrafilter $\U$ of Theorem \ref{ult} contains all sets
 $$Q^{<}_{AB}=\{i\in\II\mid |A_i|< |B_i|\},\ \textrm{for }\  |A|<|B|,$$
and this family of sets may be contained in a fine ultrafilter on $\II$ if and only if it enjoys the \FIP\ together with the family of the cones $C(d), \,d\in \II$.
Actually, this property is provable by an argument similar to that of Theorem  3.2 of \cite{BDNFuniv}.

\begin{lem}\label{whn}
Let $d\in\II$ and finitely many sets $\ A^{st}\in\WW,$  for $ s\!\in\!\! S_t,\  1\le t\le n$ 
 be given  \st\ $S_t$ is finite for all $t$, and
$ |A^{st}|=\kg_t,\  \ao\le\kg_t<\kg_u\le\kg$ for $t<u$.
Then there exists $i\in C(d)$ \st\ $|A^{st}_i|< |A^{ru}_i|$ for 
$ s\in S_t, r\in S_u,  t< u$.\
\end{lem}
\proof  Suppose chosen such an $i_m$ good  for all $t\le m$, and  let 
  
 \noindent   
    $$B_{m}=\bigcup \{A^{st}\,\mid\ t\le m\}, \ I_m=d\,\cup\!\bigcup_{b\in B_m}\!\!supp(b),\ \ k_m=\max \{|A^{st}_{i_m}|\!\mid\! s\in S_t, t\le m\}\!:$$ 
   then  pick
 $k_{m+1}>k_m$ elements with supports  not included in any $j\in I_m$
 in each set $A_{s\,m+1}\7 B_m, \ s\in S_{m+1}$, existing because $|B_m|=|I_m|=\kg_m$ and $|A_{s\,m+1}|>\kg_m.$ 
 Let $i_{m+1}$ be the union of $i_m$ with  the supports of the new elements: then 
$|A^{st}_{i_{m+1}}|=|A^
{st}_{i_{m}}|\le k_m$ for all $t\le m$, while
$|A^{s\,m+1}_{i_{m+1}}|\ge k_{m+1}>k_m$ for $s\in S_{m+1}$.

Proceeding in this way we come to a set $i=i_n\in \II$ belonging to

\smallskip 
$~~~~~~~~~~~~~~~~~~C(d)\cap\bigcap\{Q^<_{{A^{st}}{A^{ru}}}\mid  s\in S_t, r\in S_u, t<u\}.$ \qed

\bigskip
So the family $\Ql=\{Q^{<}_{AB}\mid  |A|<|B|\}\cup\{C(d)\mid d\in\II\}$ has
the \FIP, and any ultrafilter $\,\U\linc\Ql$ provides numerosities satisfying the weak Hume's principle.

 Thus we   obtain
   the reasonable effect that  the ordering of the Aristotelian numerosities refines the cardinal ordering on its universe.  
   
   We conjecture that the  family $\Ql$ shares the \FIP\  also with the family $\Hl$ of Lemma \ref{ultEu}, so it might be included in an Euclidean ultrafilter, and the consistency of the weak Hume principle \WHP\ with the difference property{ \Diff} will follow, but now this question remains open.

\section{Final remarks and open questions}\label{froq}

\subsection{Extending comparison to the whole universe $V$}\label{univ}
A simple way of extending an Aristotelian or Euclidean numerosity to the whole universe of sets $V$ could be obtained by coding the universe by ordinals, and associating to each set $X$ a set of ordinals $A_X$ in such a way that
\begin{enumerate}
  \item $A_X\cup A_Y= A_{X\cup Y}$ if $X\cap Y=\0;$
    \item $A_X\* A_Y\eq A_{X\* Y}$ if $(X\cup Y)\cap( X\*Y)=\0=X\cap Y.$
\end{enumerate}
Recalling that finite set of ordinals are appropriately coded by the \emph{natural sum of the powers} $2^\ag$ of its members, we put
$$\cg_\0=0,~~~~~\cg_{\{x_1,\ldots,x_n\}}=\bigoplus_{i=1}^n2^{\cg_{x_i}},$$
where $\cg_x$ is the ordinal coding the set $x$, 
so, in particular, each hereditarily finite set in $V_\omega$ receives its natural code in $\og$.

Now,  in order to avoid clashings with finite sets,  the codes of infinite sets
have to be chosen additively indecomposable, \ie\ pure powers of $2$ (but avoiding the so called $\eg$-numers 
$\eg=2^\eg$, which woud be confused with their singletons).
Therefore we put $\cg_\ag=2^{\ag+1}$ for each infinite ordinal $\ag$, and 
$\cg_x=\og^{\xi(x)}=2^{\og\xi(x)}$, where $\xi$ picks an ordinal in $\beth_{\ag+1}\7 \beth_\ag$ for each infinite set $x\in V_{\ag+1}\7 V_{\ag}$.
Then let 
 $$A_X=\{\cg_x\mid x\in X\},\ \ A_X*A_Y= A_{\{\{x,y\}\mid x\in X,\ y\in Y\}}=
\{\cg_{\{x,y\}}\mid x\in X, y\in Y\}.$$
Then $A_X$ is a set of ordinals, 
$A_X\cup A_Y
= A_{X\cup Y}$ if $X\cap Y=\0,$ and the ``Russellian\footnote{~
Actually Bertrand Russell warmly suggested the use of this product, which is naturally commutative and associative.}
doubleton-product" \ 
$A_X*A_Y$
might replace the Cartesian product $X\*Y$ when
$X\cap Y\cap \{\{x,y\}\mid x\in X,\ y\in Y\}=\0.$

\medskip
Let $\II$ be the class of all finite sets of ordinals, and define the counting function  $f_X: X\to\Z^\II$ by 
$f_X(i)=|(A_X)_i|=|A_X\cap i|$: then

$|(A_X)_i|+{(A_Y)_i|= |(A_{X\cup Y})_i} $ if $X\cap Y=\0$, and

$|(A_{X})_i|\cdot|(A_Y)_i|=|(A_X*A_Y)_i|$ if  
$X\cap Y\cap \{\{x,y\}\mid x\in X,\ y\in Y\}=\0$.

\medskip
When considering, instead of the universal class $V$, an initial segment $V_\kg$ with $\kg$ any inaccessible cardinal, then all works as in Theorem \ref{ult}, giving a biunique correspondence between fine ultrafilters $\U_\kg$ on $\II_\kg=[\kg]^{<\og}$ and finitely approximable numerosities $(V_\kg,\<)$, and the semiring of numerosities would become the ultrapower $\Nk_\kg=\N\ult{\II_\kg}{\U_\kg}$.

But in the case of the whole universe $V$,  the coding function $\Cg:x\mapsto\cg_x$ as well as $\II$ and $Ord$ are proper classes, so one cannot operate in \zfc, but has to work in some theory of classes, like G\"odel-Bernays theory $\textsf{GB}$ (and assume global choice, which is stronger than Zermelo's 
$\textsf{AC}$). Then one can proceed as in \cite{BDNFar}: fix an unbounded increasing sequence of cardinals $\kg$ and  ``coherent" fine ultrafilters 
$\U_\kg$ on $\II_\kg=[\kg]^{<\og}$. \Ie,  if $\kg'$ is the successor of $\kg$, then 
$\U_{\kg'}$ induces on $\II_\kg$ the equivalence $\equiv_{\U_\kg}$; at limit steps  take the union: then the class $\Nk$, direct limit of the ultrapowers
$\N\ult{\II_\kg}{\U_\kg}$, is a proper class semiring of hyperintegers suitable for assigning a numerosity to all sets of the universe. 

\subsection{The Subset Property}\label{sub}
An interesting consequence of assuming 
 $\WHP$ is the Subset Property 
 $${\SubP}~~~~~~~~~~~~~~~~~~~~~~~~ A\< B\ \ \Iff\ \ \exists A' (A\eq A'\pincl B).~~~~~~~~~~~~~~~~~~~~~~~~~~~~~~~$$
The set of numerosities 
 is very large, having the same size as the universe $W$: this  is a necessary consequence of Euclid's Principle, since one can define strictly increasing chains of sets of arbitrary length. However, any  set $A$ has only $2^{|A|}$ subsets, and so,
by assuming the Subset Property $\SubP$, instead of defining the preordering  (as we did here) through the superset property $\EP$, we would obtain that
 \emph{the initial segment of numerosities}  generated by $\nk(A)$ has size $2^{|A|}$, 
contrary, \pes, to the  large ultrapower models of \cite{BDNFar,BDNFuniv}.

This remark  proves that the set of  numerosities of sets of cardinality not exceeding $\kg$ is not forced \emph{a priori }to have cardinality exceeding $2^\kg$, 
independently of the size of the universe, by simply assuming the
Weak Hume Principle \WHP.

\subsection{The power of numerosities}\label{pow}
According to Theorem \ref{ultEu}, the power $\mk^{\nk}$ of infinite numerosities  is always well-defined, since 
numerosities are
    \emph{positive nonstandard integers.}
By using finite  approximations given by  intersections with \emph{suitable finite sets}     the interesting  relation
    $$2^{\nk(X)} = \nk([X]^{<\og})$$
    has been obtained in \cite{BDNFar}.
     Similarly, the principle \FAP\  might be adapted to provide the natural general set theoretic interpretation of powers:
   $$\mk(Y)^{\nk(X)} = \nk(\{f:X\to Y\mid |f| <\ao\}).$$ 
   \Pes\ one might assign to a finite function $f$ between sets in $\WW$ a ``support" equal to the union of the supports of the elements of domain and range of $f$.
 
 \smallskip   
The difficult problem of finding  appropriately defined arithmetic operations that give instead the numerosity of
    the \emph{function sets}  $Y^X$, or even only the \emph{full powersets }$\P(X)$, requires a quite different
    approach, and the history of the same problem for cardinalities suggests that it could not be completely solved.

\bigskip{}

\end{document}